\documentclass[12pt, a4paper,reqno]{amsart}

\usepackage[euler-digits]{eulervm}
\usepackage{graphicx,enumitem}
\usepackage{amsfonts, amsthm, amssymb, amsmath, stmaryrd}
\usepackage{mathrsfs,array}
\usepackage{eucal,times,color,accents}
\usepackage{url}
\usepackage{float}
\usepackage{pbox}
\usepackage[headings]{fullpage}

\usepackage{ytableau}
\usepackage{color}
\usepackage{mathrsfs}
\usepackage{amssymb}
\usepackage{bm}
\usepackage{amssymb}
\usepackage{ulem}
\usepackage{hyperref}
\usepackage[all,cmtip]{xy}
\usepackage{comment}

\newcommand{\ncom}{\newcommand}

\ncom{\dho}{\partial}
\ncom{\rar}{\rightarrow}
\ncom{\imply}{\Rightarrow}
\ncom{\lrar}{\longrightarrow}
\ncom{\into}{\hookrightarrow}
\ncom{\onto}{\twoheadrightarrow}
\ncom{\ov}{\overline}
\ncom{\m}{\mbox}
\ncom{\sta}{\stackrel}
\ncom{\invlim}{\varprojlim}
\ncom{\xhat}{\widehat}

\ncom{\vspc}{\vspace{3mm}}
\ncom{\End}{{\cE}nd}
\ncom{\tensor}{\otimes}

\ncom{\al}{\alpha}
\ncom{\cHom}{{\mathcal Hom}}

\ncom{\A}{{\mathbb A}}
\ncom{\comx}{{\mathbb C}}
\ncom{\E}{{\mathbb E}}
\ncom{\F}{{\mathbb F}}
\ncom{\G}{{\mathbb G}}
\ncom{\K}{{\mathbb K}}
\ncom{\Le}{{\mathbb L}}
\ncom{\N}{{\mathbb N}}

\ncom{\p}{{\mathbb P}}
\ncom{\Q}{{\mathbb Q}}
\ncom{\R}{{\mathbb R}}
\ncom{\Z}{{\mathbb Z}}

\ncom{\f}{\dfrac}

\ncom{\wtil}{\widetilde}

\ncom{\ci}{{\mathpzc i}}

\ncom{\cA}{{\mathcal A}}
\ncom{\cC}{{\mathcal C}}
\ncom{\cE}{{\mathcal E}}
\ncom{\cF}{{\mathcal F}}
\ncom{\cG}{{\mathcal G}}
\ncom{\cH}{{\mathcal H}}
\ncom{\cI}{{\mathcal I}}
\ncom{\cJ}{{\mathcal J}}
\ncom{\cK}{{\mathcal K}}
\ncom{\cL}{{\mathcal L}}
\ncom{\cM}{{\mathcal M}}
\ncom{\cN}{{\mathcal N}}
\ncom{\cO}{{\mathcal O}}
\ncom{\cP}{{\mathcal P}}
\ncom{\cQ}{{\mathcal Q}}
\ncom{\cR}{{\mathcal R}}
\ncom{\cS}{{\mathcal S}}
\ncom{\cT}{{\mathcal T}}
\ncom{\cU}{{\mathcal U}}
\ncom{\cV}{{\mathcal V}}
\ncom{\cW}{{\mathcal W}}
\ncom{\cX}{{\mathcal X}}
\ncom{\cY}{{\mathcal Y}}
\ncom{\cZ}{{\mathcal Z}}

\ncom{\fp}{{\mathfrak p}}
\ncom{\fq}{{\mathfrak q}}
\ncom{\fm}{{\mathfrak m}}
\ncom{\fn}{{\mathfrak n}}

\ncom{\cSU}{{\mathcal S \mathcal U}}
\ncom{\eop}{{\hfill $\Box$}}
\ncom{\isom}{\cong}

\DeclareMathOperator{\Spec}{Spec}

\DeclareMathOperator{\im}{Im}

\DeclareMathOperator{\Ext}{Ext}

\DeclareMathOperator{\Tor}{Tor}
\DeclareMathOperator{\Hom}{Hom}
\DeclareMathOperator{\Tr}{Tr}
\DeclareMathOperator{\depth}{depth}
\DeclareMathOperator{\height}{ht}

\DeclareMathOperator{\Ass}{Ass}
\DeclareMathOperator{\Gdim}{G-dim}
\DeclareMathOperator{\pd}{pd}

\theoremstyle{plain}
\newtheorem{theorem}{Theorem}
\newtheorem{lemma}[theorem]{Lemma}

\newtheorem{conj}{Conjecture}

\newtheorem{corollary}[theorem]{Corollary}
\newtheorem{proposition}[theorem]{Proposition}
\theoremstyle{definition}

\newtheorem{defn}{Definition}

\theoremstyle{remark}
\newtheorem{remark}{Remark}

\long\def\comment#1{}

\newtheorem{exmp}{Example}[section]

\theoremstyle{plain}
\newtheorem{introtheorem}{Theorem}

\newtheorem*{recalledA}{Theorem~\ref{WL_tensor product_cyclic}}

\newtheorem*{recalledB}{Theorem~\ref{Auslander-Reiten_torsionfree}}

\newtheorem*{recalledC}{Theorem~\ref{thm_C}}

\begin{document}

\title[Weak and Serre Lifting]{Weak and Serre Lifting of Cyclic Modules and the Liftability of Auslander Transposes}

\author{Shashi Ranjan Sinha}
\address{Department of Mathematics, Indian Institute of Technology -- Hyderabad, 502285, India.}
\email{ma20resch11005@iith.ac.in}

\author{Amit Tripathi}
\address{Department of Mathematics, Indian Institute of Technology -- Hyderabad, 502285, India.}
\email{amittr@gmail.com}

\begin{abstract} 
	Let \(Q\) be a local ring, \(f\) a nonzerodivisor, and \(R=Q/(f)\). We characterize weak and Serre liftability of cyclic $R$-modules under suitable hypotheses, obtaining complete-intersection criteria for Serre liftability in small height. We prove that the tensor product of two weakly liftable cyclic modules remains weakly liftable when their first Tor module vanishes, and give sufficient conditions for tensor products to preserve Serre liftability. We construct an Artinian Gorenstein quotient \(Q/I\), with \(Q\) regular, that is perfect over $R$ and Serre liftable but not weakly liftable to \(Q\), showing that the negative answer to a question of Jorgensen persists in this more restrictive setting. Finally, we characterize when a lift induces a lift of the Auslander transpose and obtain a freeness criterion related to the Auslander–Reiten conjecture, recovering a theorem of Ghosh and Samanta.
\end{abstract}

\maketitle

\section{Introduction}

Let $(Q, \fn, k)$ be a local ring, let $f \in \fn$ be a nonzerodivisor, and set $R = Q/(f)$. We study lifting of $R$-modules along the quotient map $Q \to R$. One motivation for lifting of modules comes from Grothendieck's lifting problem, which asked whether every $R$-module lifts to $Q$ in the case where $Q$ is a complete regular local ring and $f\in\fn\setminus\fn^2$. 

For a given $R$-module $M$,  the classical liftability \cite{BE2} requires the existence of a $Q$-module $N$ such that $N \otimes_Q R \cong M$ and $Tor_i^Q(R, N) = 0$ for all $i > 0$.  Since then two less restrictive notions of liftability have been investigated. Auslander, Ding, and Solberg \cite{ADS} introduced the concept of weak liftability, and KC and Soto Levins \cite{NL} have recently introduced Serre liftability and proved Serre’s positivity conjecture over ramified regular local rings for such modules. 

In this paper, we investigate the weak and Serre lifting properties of cyclic modules and their tensor products.  For cyclic modules, lifting conditions can be expressed directly in terms of defining ideals. This approach is inspired by the emphasis of Buchsbaum–Eisenbud \cite{BE2} and Hochster \cite{Ho} on cyclic modules in Grothendieck’s lifting problem and Serre’s multiplicity conjecture. 

Beyond the cyclic setting, we also examine the liftability of Auslander transpose of finitely generated modules, with an application to the Auslander–Reiten conjecture. We now give an outline of the paper.

Suppose that $Q$ is Cohen–Macaulay, $f$ is a minimal generator of $I$, and $h = \height(I)$. We prove in Section 3 that $Q/I$ is Serre liftable to $Q$ if and only if there exists a prime ideal $\fq$ such that
$$\height(\fq) = h-1, \qquad I\subseteq\mathfrak q+(f).$$ 
This gives complete-intersection criteria in small height; in particular, for a height-two prime ideal in a Cohen–Macaulay local UFD, Serre liftability is equivalent to being a complete intersection. Furthermore, we prove that the tensor product of Serre liftable modules over a regular local ring retains Serre liftability under specific Tor-vanishing conditions.   

Section 4 focuses on weak liftability. We provide equivalent conditions for the weak liftability of cyclic modules and establish the following weak lifting counterpart to Jorgensen's liftability criterion for tensor products. 

\begin{introtheorem} \label{WL_tensor product_cyclic} Let $I_1$ and $I_2$ be ideals of $R$ such that $\Tor_1^R(R/I_1, R/I_2) = 0$. If $R/I_1$ and $R/I_2$ are weakly liftable to $Q$, then so is $R/(I_1 + I_2)$.  
\end{introtheorem} 

In Section 5, we study $n$-torsionless modules. Given a lift $N$ of $M$, we characterize when $\Tr_Q(N)$ is a lift of $\Tr_R(M)$. Applying these results to the Auslander-Reiten conjecture, we obtain the following freeness criterion, which replaces the finiteness hypothesis on  Gorenstein-dimension in  \cite[Theorem~3.10]{GS} by torsionlessness conditions and recovers that theorem.

\begin{introtheorem} \label{Auslander-Reiten_torsionfree}
	Let $R$ be a local ring and let $M$ be an $R$-module such that $CI\text{-}dim_R Hom(M, M) < \infty$. Assume that $\Ext_R^1(M, R) = 0$ and $\Ext_R^i(M, M) = 0$ for all $i \ge 1$. If $M$ is $n$-torsionless for all positive integers $n$, then $M$ is free. 
\end{introtheorem} 

Finally, Section~6 contains examples, including our principal counterexample. In \cite[Question~3.5]{J1}, Jorgensen asked if every perfect cyclic module lifts along $Q \to Q/(f)$. Smoke \cite[Section~6]{Smoke} constructed a perfect cyclic module with negative intersection multiplicity; as observed in \cite[Example~2.1.7]{NKC}, the intersection-multiplicity obstruction shows that this module is not liftable to the ambient regular ring. Using an Artinian Gorenstein algebra introduced by Gasharov and Peeva \cite{GP}, we construct a perfect cyclic module that is Serre liftable but not weakly liftable to the ambient regular local ring. This example complements \cite[3.3]{NL} by exhibiting failure of weak liftability in the cyclic Gorenstein setting.

\begin{introtheorem} \label{thm_C}
	There exists a perfect cyclic module of codimension $4$ over a hypersurface singularity that admits a Serre lift but fails to be weakly liftable to the ambient regular local ring. The module can be chosen to be an Artinian Gorenstein quotient of a $5$-dimensional complete regular local ring.
\end{introtheorem}

Moreover, the ambient codimension of this example is minimal among Gorenstein quotients of finite projective dimension in the sense explained in Remark \ref{rem_6}.

\section{Preliminaries} \label{prelim} 

Throughout, rings are commutative, Noetherian and local, and modules are finitely generated. For an $A$-module $L$, we write $L^*=\Hom_A(L,A)$, with the ring understood from the context. 

Let $(Q,\fn,k)$ be a local ring and let $R = Q/(f)$, where $f \in \fn$ is a nonzerodivisor. An $R$-module $M$ is said to \textit{lift} to $Q$ (or simply $M$ is \textit{liftable}) if there exists a $Q$-module $N$ such that $N \otimes_{Q} R \cong M$ and $\Tor_i^{Q} (R, N) = 0$ for all $i > 0;$ see \cite{BE, Dao, Ho,J1, J2}. 

\subsection{Weak lifting} \label{sec_WL}
In \cite{ADS}, Auslander, Ding, and Solberg introduced and studied the notion of weak lifting.
\begin{defn}
	Let $Q \rar R$ be a homomorphism of rings, and $M$ be an $R$-module. Then, $M$ is said to \textit{weakly lift} (or \textit{weakly liftable}) to $Q$ if it is a direct summand of a liftable $R$-module.
\end{defn} 

It is clear that a liftable $R$-module is weakly liftable. The converse is not true, for instance, see \cite{ADS} for a counterexample. We now restrict our attention to the setup with $R = Q/(f)$, as defined earlier.  In this context, we briefly review some results on weak liftability from \cite{ADS}. Consider the short  exact sequence of $Q$-modules
\[0 \rar  \Omega_1 ^Q M \rar Q^n \rar M \rar 0,\] 
where $\Omega_1 ^Q M$ is the first syzygy of $M$ defined by the minimal free resolution of $M$ over $Q$. Tensoring this sequence with $R$ and using the isomorphism $\Tor_{1} ^{Q} (M,R) \cong M$ one obtains
\[
0 \rar M \xrightarrow{\alpha} \Omega_1 ^Q M /f\Omega_1 ^Q M \rar R^n \rar M \rar 0.
\]

This $4$-term sequence induces the following short exact sequence
\begin{equation}\label{WL_exact sequence}
	0 \rar M \xrightarrow{\alpha} \Omega_1 ^Q M /f\Omega_1 ^Q M \rar \Omega_1 ^R M \rar 0.
\end{equation}

We record the following results:
\begin{enumerate}[resume, label = \normalfont(\alph*)]
	\item 	\label{WL_split map} $M$ is weakly liftable to $Q$ if and only if the map $\alpha$ splits; see \cite{ADS}.
	\item \label{corollary_WL_cyclic module} 	If ${I}$ is an ideal of $Q$ containing $f$, then $Q/{I}$ weakly lifts to $Q$ if and only if the map $\alpha: Q/I \rar I/fI$ defined by $r + I \mapsto fr + fI$ is split.
	\item \label{Dao_criterion} (Dao's necessary condition) If $Q/I$ is weakly liftable, then the induced map $Q/I \rar I/I^2$ is an inclusion of a direct summand. Thus, it holds necessarily that $(I^2: f) \subset I$. For a more general version, see \cite[Theorem~3.2]{Dao}. 
\end{enumerate}

\begin{remark}\label{not WL_residue field}
	It follows that if $Q/I$ is weakly liftable to $Q$, then $f$ must be a minimal generator of ${I}$. Indeed, if $f \in \fn{I}$,  the weak liftability of $Q/{I}$ to $Q$ yields a split exact sequence $0 \rar Q/{I} \rar {I}/f{I} \rar {I}/(f) \rar 0$ by item \ref{corollary_WL_cyclic module} above. However, tensoring this sequence with $Q/\fn$ yields a contradiction.
\end{remark}

\subsection{Serre Liftable modules} 
KC and Soto Levins \cite{NL} have introduced Serre liftability to study questions concerning lengths and multiplicities arising from Serre’s intersection multiplicity pairing. 

\begin{defn}\label{Serre lift_defn}
	Let $Q \rar R$ be a surjective homomorphism of local rings. An $R$-module $M$ is said to be \textit{Serre liftable} to $Q$ if there exists a $Q$-module $N$ such that:
	\begin{enumerate}
		\item $M \cong N \otimes_{Q} R$;
		\item $\dim Q -\dim N = \dim R -\dim M$.
	\end{enumerate}
\end{defn}

If $R$ is a codimension $t$ deformation of $Q$, it is easy to see that any lift of $M$ is also a Serre lift. When $Q$ is regular, the same was proved \cite[Proposition~1.6]{NL} for any arbitrary surjection $Q \rar R$. The converse, however, fails as was shown in \cite[Proposition~3.3]{NL} using the notion of cohomological operators. In Examples \ref{exmp_main} and \ref{exmp_main2} we provide instances of cyclic $R$-modules that are Serre liftable to $Q$ but not weakly liftable. 

We record the following elementary height inequality.

\begin{lemma}\label{rem_2}
	Let $(Q, \fn)$ be a local Cohen-Macaulay ring, and let $J \subseteq Q$ be an ideal. Then for any $f \in \fn$, we have
$$\height({J}+(f)) \leq \height({J}) + 1.$$ 
\end{lemma}

\begin{proof}
	Since $Q$ is a local Cohen-Macaulay ring, it follows by the dimension equality that
	$$\height(J) + \dim(Q/J) =\height(J + (f)) + \dim(Q/(J+(f))).$$
Thus, it suffices to show that $
	\dim(Q/J) \le \dim(Q/(J + (f))) + 1.$ But this inequality clearly holds as $f \in \fn$; see  \cite[Appendix~A.5]{BH}.
\end{proof}

\subsection{Auslander transpose} \label{sec_auslander} A general reference for the results in this subsection is \cite{AB} and a gentler introduction can be found in \cite{Masek}.

Let $R$ be a local ring and let $M$ be a finitely generated $R$-module. Let $F_1 \rar F_0 \rar M \rar 0$ be a free presentation of $M$.  The Auslander transpose of $M$, denoted by $\Tr_R(M)$,  is the cokernel of the induced map $F_0^* \rar  F_1^*$. 

In general, the Auslander transpose is uniquely determined only up to projective equivalence. However, because we are working over a local ring, minimal free presentations are unique up to isomorphism. Therefore, by defining $\Tr_R(M)$ via a minimal free presentation of $M$, we may regard the transpose as being uniquely determined up to isomorphism.

We say that $M$ is $n$-\textit{torsionless} if $\Ext^i_R(\Tr_R(M),R)=0$ for all $1 \leq i \leq n$. We say $M$ is an $n$-th syzygy if there exists an exact sequence $0 \rar M \rar F_0 \rar F_1 \rar \cdots \rar F_{n-1}$ where $F_i$ are projective.  The finitely generated $R$-module $M$ is said to satisfy the $\wtil{S}_n$ property if $\depth(M_{\fp}) \geq \min\{n, \depth(R_{\fp})\}$ for all $\fp\in \Spec(R)$; see \cite{Masek}.

\section{Serre lifting of cyclic modules}
	Throughout this section, unless stated otherwise, we assume the following: $Q$ is a local ring, and $R=Q/(f)$ where $f$ is a $Q$-regular element. For any ideal ${K} \subset Q,$ let $\min(K)$ be the set of minimal primes over $K$ and let $\min_0(K)$ be the subset of those minimal primes $\fp$ over $K$ for which $ht(K) = ht(\fp)$. By Noetherian assumption, $\min(K)$ and hence $\min_0(K)$ is a finite set. If $Q$ is Cohen-Macaulay, and $Q/K$ is equidimensional, then it follows by \cite[Corollary~2.1.4]{BH} that $\min(K) = \min_0(K)$.

\begin{lemma} \label{lemma_serre_gen_equiv} Assume $Q$ is Cohen-Macaulay and let $J$ be an ideal. Set ${I} = {J} + (f)$. Then the following are equivalent: 
	\begin{enumerate}[label=\normalfont(\alph*)]
		\item  \label{lemma_serre_gen_equiv_a}  The $Q$-module $Q/J$ is a Serre lift of the $R$-module $Q/I$. 
		\item  \label{lemma_serre_gen_equiv_b}  $\height({I}) = \height({J}) + 1$. 
		\item \label{lemma_serre_gen_equiv_b'}  $grade({I}) = grade({J}) + 1$. 
		\item  \label{lemma_serre_gen_equiv_c}  $\min_0({I}) \cap \min_0({J}) = \emptyset$.
		\item  \label{lemma_serre_gen_equiv_d} $f \notin \bigcup_{\fp\in \min_0({J})} \fp$. 
	\end{enumerate} 
\end{lemma}
\begin{proof} 
Since $Q$ is local Cohen-Macaulay, we have \begin{align*}
		\height({I}) - \height({J}) & =\left( \dim(Q) -\dim(Q/{I})\right) - \left(\dim(Q) - \dim(Q/{J}) \right)\\ & =  1+\left(\dim(R) -\dim(Q/{I})\right) - \left(\dim(Q) - \dim(Q/{J}) \right).
	\end{align*} This proves $\ref{lemma_serre_gen_equiv_a} \iff \ref{lemma_serre_gen_equiv_b}$. The equivalence $\ref{lemma_serre_gen_equiv_b} \iff \ref{lemma_serre_gen_equiv_b'}$ follows from \cite[Corollary~2.1.4]{BH}.  
	
	If $\fq \in \min_0({I}) \cap \min_0({J}),$ then clearly $\height({I}) = \height(\fq) = \height({J}),$ a contradiction.  So $\ref{lemma_serre_gen_equiv_b} \implies \ref{lemma_serre_gen_equiv_c}$. To see $\ref{lemma_serre_gen_equiv_c} \implies \ref{lemma_serre_gen_equiv_b},$ let $\fp \in \min_0({I})$. Then ${J} \subseteq \fp$ and $\fp \notin \min_0({J})$. Thus, $\height({I}) = \height(\fp) > \height({J})$. Since ${I} = {J} + (f),$ height can increase at most by $1$ by Lemma \ref{rem_2}, so we get $\ref{lemma_serre_gen_equiv_b}$. 
	
	It is clear that $\ref{lemma_serre_gen_equiv_d} \implies \ref{lemma_serre_gen_equiv_c}$. To complete the proof, we show that $\ref{lemma_serre_gen_equiv_b} \implies \ref{lemma_serre_gen_equiv_d}$. If $f \in \fp$ for some $\fp \in \min_0({J}),$ then ${I} \subseteq \fp$. Thus, $\height({I}) = \height({J}),$ which contradicts $\ref{lemma_serre_gen_equiv_b}$. 
\end{proof}

We now point out several applications of Lemma~\ref{lemma_serre_gen_equiv} starting with a characterization for Serre lifting of cyclic modules. 

\begin{corollary} \label{lemma_char_serre} 
	Assume that $Q$ is Cohen-Macaulay. Let $I$ be an ideal and assume that $f$ is a minimal generator of ${I}$. Let $h$ be the height of $I$. Then the following are equivalent
	\begin{enumerate}[label=(\alph*)]
		\item The cyclic $R$-module $Q/{I}$ is Serre liftable.
		\item There is a prime ideal $\fq$ of height $h-1$ such that ${I} \subseteq \fq + (f)$. 
	\end{enumerate} If any of these conditions hold, then there is an ideal $J \subset \fq$ of height $\height(I) - 1$ such that $Q/J$ is a Serre lift of $Q/I$ to $Q$. 
\end{corollary}
\begin{proof}
	Let $Q/J$ be a Serre lift of $Q/I$ and let $\fq \in min_0({J})$. Then by Lemma~\ref{lemma_serre_gen_equiv}, $\height(\fq) =  \height(J) = h-1$ and ${I} = J+(f) \subseteq \fq + (f)$.  
	
	Conversely, suppose $\fq$ is as given. Let $(f, f_2,\cdots, f_m)$ be a minimal generating set of ${I}$. Let $g_i \in \fq$ and $a_i \in Q$ be such that $f_i  = g_i + a_if$ for $i = 2,\cdots, m$. Set ${J} = (f_2 - a_2f, \cdots, f_m - a_mf)$. Then ${J} \subseteq \fq$ and ${J} + (f) = {I}$. By Lemma \ref{rem_2}, $\height({J}) \geq h-1$ but since ${J}$ is contained in a height $h-1$ prime ideal, we must have $\height({J}) = h-1$. Thus, by Lemma~\ref{lemma_serre_gen_equiv}, $Q/{J}$ is a Serre lift of $Q/{I}$. The last assertion is clear.  
\end{proof}

As an immediate application, we show that under suitable assumptions, Serre liftable ideals of small height are precisely the complete intersections.

\begin{corollary} \label{cor_small_height}
	Let $Q$ be a local Cohen-Macaulay domain and  $R = Q/(f)$ where $f$ is a nonzerodivisor. Let $I$ be an ideal of $Q$ such that $f$ is a minimal generator of ${I}$.
	\begin{enumerate}[label=\normalfont(\alph*)]
		\item \label{item_cor_sm_a}   If $\height(I) = 1,$ then the cyclic $R$-module $Q/I$ is Serre liftable to $Q$ if and only if $I$ is the principal ideal $(f)$. 
		\item \label{item_cor_sm_b}    If $\height(I) = 2,$ $I$ is prime and $Q$ is, in addition, a unique factorization domain, then $Q/I$ is Serre liftable to $Q$ if and only if $I$ is a complete intersection ideal. 
	\end{enumerate}
\end{corollary}
\begin{proof} For \ref{item_cor_sm_a}, Corollary \ref{lemma_char_serre} and the domain hypothesis give \(J=0\), hence \(I=(f)\).
	
	To see \ref{item_cor_sm_b}, suppose $I \subseteq \fq + (f)$. Since  $\height(\fq+(f)) = \height(I),$ and $I$ is prime, we must have $I = \fq+(f)$. Since $\fq$ is a height $1$ prime in a unique factorization domain, it is principal, thus $I$ is a complete intersection ideal. 
	
	The converse implication holds more generally in both cases, as $Q/I$ clearly lifts when $I$ is a complete intersection ideal. 
\end{proof}

Corollary \ref{cor_small_height} can be compared with the low-height restrictions on weakly liftable ideals in \cite[Corollary~3.7]{Dao}.

\begin{corollary} \label{cor_1} With the same setup as in Lemma~\ref{lemma_serre_gen_equiv}, assume that $Q/J$ is a Serre lift of the $R$-module $Q/I$ to $Q$. Furthermore, assume that one of the following holds.
	\begin{enumerate}[label=\normalfont(\alph*)]
		\item \label{item_cor_c} $Q/{J}$ is equidimensional, and ${{J}}$ is a radical ideal.
		\item  \label{item_cor_d}  $Q/J$ is Cohen-Macaulay. 
	\end{enumerate} Then $Q/{J}$ is a lift of $Q/{I}$ to $Q$. 
\end{corollary}
\begin{proof}We first prove \ref{item_cor_c}. By equidimensionality assumption, $\min_0({J}) = \min({J})$. Thus, for any $\fp \in \min({J}) = \min_0(J),$  we must have $f \notin \fp$ by  Lemma~\ref{lemma_serre_gen_equiv}. Therefore, $({J}:f) \subseteq \bigcap_{\fp \in \min(J)} \fp = \sqrt{J} = J$.  The proof now follows from the characterization that the liftability of the cyclic $R$-module $Q/{I}$ to $Q$ is equivalent to existence of an ideal ${J} \subsetneq {I}$ such that ${I} = {J} + (f)$ and $({J}:f) = {J};$ see \cite[p.~458]{Ho}.  
	
The proof of \ref{item_cor_d} follows from the fact that $\Ass(Q/J) = \min(J) = \min_0(J)$, thus by Lemma~\ref{lemma_serre_gen_equiv}, $f$ is a nonzerodivisor on $Q/J$. 
\end{proof} 

\begin{exmp}
	Let $Q = k[[x,y,z]],$ and ${J} = (xy,xz)$. Set $f = y,$ and  ${I} = {J} + (f) = (xz, y),$ it follows by Lemma~\ref{lemma_serre_gen_equiv} that  $Q/{J}$ is a Serre lift of $Q/{I}$ to $Q$. 	Since \(J=(x)\cap(y,z)\), we have \(\min_0(J)=\{(x)\}\). The quotient \(Q/I\) is a complete intersection, whereas \(Q/J\) is not equidimensional. Moreover, \((J:f)=(x)\supsetneq J\), so this Serre lift is not an ordinary lift. Thus, in Corollary \ref{cor_1}, the assumptions on $Q/J$ cannot be replaced by the Cohen–Macaulayness of $Q/I$, even when $J$ is radical.

\end{exmp}

\begin{remark} \label{rem_1} 
	Let $Q/{J}$ be a Serre lift of $Q/{I}$. For any $\fq \in \min_0({J}),$ there is a $\fp \in \min_0({I}),$ such that $\fq \subset \fp$. This follows as ${I} \subset (\fq,f)$ and $\height({I}) = \height(\fq,f)$. So $\emptyset \neq \min_0(\fq,f) \subset \min_0({I})$ and we can choose any $\fp \in \min_0(\fq,f)$. 
\end{remark}

\subsection{Serre liftability of tensor products} Serre liftability was originally defined in \cite{NL} for regular local rings. In this subsection, we also assume that $Q$ is a regular local ring. Let ${\bf f} = \{f_1,\cdots, f_t\} \subset Q$ be a regular sequence. Define $R = Q/({\bf f})$. Let $M_1$ and $M_2$ be two $R$-modules that are Serre liftable to $Q$-modules $N_1$ and $N_2$. It follows (see, for instance, \cite[Appendix~A.4]{BH}) that
\begin{equation} \label{eqn_ineq} 
	\dim(M_1 \otimes_R M_2)  \geq \dim(N_1\otimes_Q N_2) - t.
\end{equation}

\begin{lemma} \label{lemma_serre_inequality} 
Assume the setup above. Then $M_1$ and $M_2$ satisfy the Serre inequality, i.e., $$\dim(M_1) + \dim(M_2) \leq \dim(R) + \dim(M_1 \otimes_R M_2),$$ with equality only if $\dim(N_1 \otimes N_2) - \dim(M_1 \otimes M_2) = t$. 
\end{lemma}
\begin{proof}
	By Serre liftability, we know that $\dim(M_i)  = \dim(N_i) - t$ for $i = 1,2$.  Using Serre's inequality \cite{Serre}, we have that $\dim(N_1) + \dim(N_2) \leq \dim(Q) + \dim(N_1 \otimes N_2)$.  Combining these facts, along with the inequality \eqref{eqn_ineq}, we obtain $$\dim(M_1) + \dim(M_2) \leq \dim(R) +  \dim(M_1 \otimes M_2).$$ It is clear that equality implies that $\dim(N_1 \otimes N_2) - \dim(M_1 \otimes M_2) = t$. 
\end{proof}

\begin{remark}
	Let $Q$ be a regular local ring and let $R$ be a locally complete intersection $Q/({\bf f})$ and assume  that $R/\overline{I}_1$ and $R/\overline{I}_2$ are Serre liftable to $Q$. It follows by Lemma~\ref{lemma_serre_inequality} that $\dim(R/\overline{I}_1) + \dim(R/\overline{I}_2) \leq \dim(R) + \dim(R/(\overline{I}_1+\overline{I}_2))$. Since $R$ is local CM,  we get $$\height(\overline{I}_1 + \overline{I}_2) \leq \height(\overline{I}_1) + \height(\overline{I}_2).$$ 
\end{remark}

\begin{lemma} \label{lemma_serre_tensor_prod} 
Let $Q$ be a regular local ring and let ${\bf f} = \{f_1,\cdots, f_t\} \subset Q$ be a regular sequence. Set $R = Q/({\bf f})$. Let $M_1$ and $M_2$ be two $R$-modules that are Serre liftable to $Q$-modules $N_1$ and $N_2$. Then $N_1 \otimes N_2$  is a Serre lift of $M_1 \otimes M_2$ to $Q$ if any of the following holds. 
	\begin{enumerate}[label=\normalfont{(\alph*)}]
		\item \label{item_1a} Either $M_1$ or $M_2$ is supported everywhere on $\Spec(R)$. 
		\item \label{item_2a} $\ell_R(M_1 \otimes M_2) < \infty,$ and $\dim(M_1) + \dim(M_2) = \dim(R)$. 
		\item \label{item_3a} $M_1 \otimes M_2$ is {CM} and $\Tor_i^R(M_1,M_2) = 0$ for $i = 1,\cdots, t+1$. 
	\end{enumerate}  Furthermore, if \ref{item_3a} holds, then we also conclude that $M_1$ and $M_2$ are CM. 
\end{lemma}
\begin{proof} Since the map $Q \rar R$ is a codimension $t$ deformation, to show Serre liftability of $M_1 \otimes_R M_2,$ it is sufficient to show that $\dim(N_1\otimes_Q N_2) - \dim(M_1 \otimes_R M_2) = t$.
	
	Without loss of generality, we assume that $Supp(M_1) = Spec(R),$ which implies that $\dim(M_1) = \dim(R)$.  Thus,  by Lemma~\ref{lemma_serre_inequality}, $\dim(M_2) \leq \dim(M_1 \otimes_R M_2)$. This is clearly an equality, hence, all the above inequalities are equalities. In particular, $\dim(N_1\otimes_Q N_2) - \dim(M_1 \otimes_R M_2) = t$. This proves \ref{item_1a}. 
	
	The proof of \ref{item_2a} is similar - the hypothesis implies that the inequality in Lemma~\ref{lemma_serre_inequality} is an equality. Thus,  we must have $ \dim(N_1 \otimes N_2)  = t$.

	To see \ref{item_3a}, note that it follows by \cite{Murthy} that $M_1$ and $M_2$ are Tor-independent. Thus, by Auslander's depth formula \cite[Theorem~2.5]{HW} \begin{align*}
		\depth(M_1 \otimes M_2) &= \depth(M_1) + \depth(M_2) - \depth(R) \leq \dim(M_1) + \dim(M_2) - \dim(R) \\ & = \dim(N_1) + \dim(N_2) -\dim(Q) - t \leq \dim(N_1 \otimes N_2) - t \\ & \leq \dim(M_1 \otimes M_2).
	\end{align*} Since the first and last terms are equal, we conclude that all the inequalities are equalities. In particular, the modules $M_i$ are Cohen-Macaulay for $i = 1,2$, and $M_1 \otimes_R M_2$ is Serre liftable.
\end{proof}

\begin{remark}
	In Lemma \ref{lemma_serre_tensor_prod}, if $t = 1$ and $R$ is an \textit{admissible} hypersurface, and $M_1$ and $M_2$ are Serre liftable, such that $\ell_R(M_1 \otimes_R M_2) < \infty,$ then by Lemma \ref{lemma_serre_inequality}, they \textit{intersect decently}; see \cite{Dao2}.  This implies that Hochster's theta invariant $\theta_R(M_1,M_2)$ vanishes, and the pair $(M_1,M_2)$ is Tor-rigid \cite[Proposition~2.8]{Dao2}. In particular, in Lemma~\ref{lemma_serre_tensor_prod} \ref{item_3a}, we require only first Tor vanishing. 
\end{remark}

\section{Weak lifting of cyclic modules}

Throughout this section, we assume the following: $Q$ is a local ring, $I$ is an ideal, and $f \in I$ is a minimal generator that is also a $Q$-regular element.  We define $R=Q/(f)$. 

\begin{lemma}\label{lemma_exist} The following are equivalent.
	\begin{enumerate}[label=\normalfont(\alph*)]
		\item \label{item_a} The cyclic $R$-module $Q/{I}$ is weakly liftable to $Q$.
		\item \label{item_b}  The induced map $\beta: Q/{I} \rar {I}/{I}^2$ is split.
		\item \label{item_c}  There is an ideal ${J} \subset {I}$ such that ${I} ={J} + (f)$ and $({J}:f) \subseteq {I}$. 
	\end{enumerate} 
\end{lemma}
\begin{proof}
	The equivalence $\ref{item_a}  \iff \ref{item_b}$ was shown in \cite[Lemma~4.1]{Dao}. Let $\alpha: Q/I \rar I/fI$ be the  map as in the characterization \ref{WL_split map} in \ref{sec_WL} and $\theta: {I}/f{I} \rar Q/{I}$ be a splitting map. Let $\{f, f_2,\cdots, f_n\}$ be a minimal generating set of ${I}$ and assume that $q_i \in Q$ be such that $\theta(f_i) = q_i + {I}$ for $i = 2,\cdots, n$. Set $g_i = f_i - q_i f$. It is easy to see that ${I}  = (f, g_2,\cdots, g_n)$ and $\theta(g_i) = 0$ for $i = 2,\cdots, n$. Set ${J}  = (g_2,\cdots, g_n)$. The well definedness of the map $\theta$ is equivalent to the containment $({J}:f) \subseteq {I}$. This shows that $\ref{item_a}\implies \ref{item_c}$. 
	
	Conversely, assume \ref{item_c}. Define $\theta: {I}/f{I} \rar Q/{I}$ as $\theta(j + af + f {I}) = a + {I}$ for $j \in J$ and $a \in Q$. The map $\theta$ is well-defined by the hypothesis. It is also clear that $\theta$ splits $\alpha: Q/{I} \rar {I}/f{I}$. This proves \ref{item_a}. 
\end{proof}

To summarize, under the standing hypotheses, and assuming that \(Q\) is Cohen–Macaulay, the three lifting conditions can therefore be compared as follows. In each case one seeks an ideal \(J\subseteq I\) with \(I=J+(f)\); the additional requirement is
\[
\begin{array}{ll}
	\text{ordinary lifting:}&(J:f)=J,\\[2pt]
	\text{weak lifting:}&(J:f)\subseteq I,\\[2pt]
	\text{Serre lifting:}&\operatorname{ht}J=\operatorname{ht}I-1.
\end{array}
\]
The ideal \(J\) may be different in the three characterizations. 

The following example shows that in our more general setting, weak liftability does not imply Serre liftability (though we note that Serre liftability was defined only for the case where $Q$ is a regular local ring in \cite{NL}). 
\begin{exmp}\label{WL_one dimensional domain} 
	Let $Q=k[[X,Y]]/(X^2 -Y^3)$, where $k$ is a field, and let $x$ and $y$ be the images of $X$ and $Y$ in $Q,$ respectively. Set $f=x,$ $R=Q/(f),$ and  ${I}=(x,y^2)$. Taking $J = (y^2),$ we can easily verify that $(J:f) \subseteq I$ and $I = J+(f)$. Thus, by Lemma~\ref{lemma_exist}, $Q/{I}$ weakly lifts to $Q$. Since $Q$ is a Cohen-Macaulay domain, and $\height(I) = 1,$ it follows by Corollary~\ref{cor_small_height} that $Q/{I}$ is not Serre liftable to $Q$.
\end{exmp}

\subsection{Weak liftability of tensor products} 

In this subsection, we focus primarily on the module structure over $R$. For ease of notation, we drop the overbar convention for ideals and denote an ideal of $R$ as $I$. We use $\wtil{I}$ to denote the corresponding preimage in $Q$.

Jorgensen \cite[Corollary~2.5]{J1} established preservation of liftability under tensor products of cyclic modules when the first Tor module vanishes. Using a gluing argument, we prove Theorem~\ref{WL_tensor product_cyclic}, which establishes a similar result for weak liftability of tensor products of cyclic modules. For convenience, we recall its statement before giving the proof.

\begin{recalledA}
	Let ${I}_1$ and ${I}_2$ be ideals of $R$ such that $\Tor_1^R (R/ {I}_1,R/ {I}_2)=0$. If  $R/{I}_1$ and $R/{I}_2$ are weakly liftable to $Q$, then so is  $R/({I}_1 +{I}_2)$.
\end{recalledA}
\begin{proof}
	Let $\wtil{I_1}$ and $\wtil{I_2}$ be the preimage ideals in $Q$ of $I_1$ and $I_2$, respectively. Since $R/I_1$ and $R/I_2$ weakly lift to $Q$, it follows by  Section \ref{sec_WL}  item \ref{corollary_WL_cyclic module} that the following sequences are split,
	\begin{align*} \label{eqn_0}
		\delta_1:\ \ 	0 \rar R/I_1 \rar \wtil{I_1}/f\wtil{I_1} \xrightarrow{p_1} I_1 \rar 0 \\
		\delta_2:\ \ 	0 \rar R/I_2 \rar \wtil{I_2}/f\wtil{I_2} \xrightarrow{p_2} I_2 \rar 0.
	\end{align*} Let $\pi_j: I_j \rar  \wtil{I_j}/f\wtil{I_j}$ be the splitting maps, i.e., $p_j \circ \pi_j = 1$ for $j = 1,2$. Set \[
E=(\widetilde I_1+\widetilde I_2)/
f(\widetilde I_1+\widetilde I_2).
\] Consider the sequence \[0 \rar I_1 \cap I_2 \rar  I_1 \oplus I_2 \rar I_1 + I_2 \rar 0.\]  Applying  $\Hom_R(-,E)$ to it,  we get 
	\begin{align} \label{eqn_2} 
		0 \rar \Hom_R(I_1 + I_2,E ) & \xrightarrow{h'} \Hom_R(I_1 \oplus I_2,E )  \xrightarrow{h} \Hom_R(I_1 \cap I_2,E ).
	\end{align} Consider the map $\alpha \in \Hom_R(I_1 \oplus I_2,E )$ defined as \[\alpha(i_1, i_2) = g_1\circ \pi_1(i_1) + g_2\circ \pi_2(i_2),\] where, for $j = 1,2,$ $$g_j:  \wtil{I_j}/f\wtil{I_j} \rar E $$ is the natural map induced by the composition $$ \wtil{I_j} \rar \wtil{I_1}+\wtil{I_2} \rar E .$$ 
	
	The image of $\alpha$ under $h$ is the map $h(\alpha): I_1 \cap I_2 \rar E $ given as $h(\alpha)(i) = \alpha(i,-i)$. We claim that under our hypothesis, the map $h(\alpha) = 0$. To see this, let $i_1i_2 \in I_1I_2 = I_1 \cap I_2$ be any element. Then \begin{align} \label{eqn_1} 
		h(\alpha)(i_1i_2) = \alpha(i_1i_2, -i_1i_2) = g_1\circ \pi_1(i_1i_2) -g_2\circ \pi_2(i_1i_2).
	\end{align} The element $i_1i_2$ has a lift $\wtil{i_1}\wtil{i_2} \in Q,$ where  $\wtil{i_1} \in \wtil{I_1},$ and $\wtil{i_2} \in \wtil{I_2}$.

	Set $\pi_j(i_j) = \wtil{x_j} + f\wtil{I_j} \in \wtil{I_j}/f\wtil{I_j},$ for $j = 1,2$. Since the map $\wtil{I_j} \rar I_j$ factors via $\wtil{I_j}/f\wtil{I_j},$ and by splitting, $p_j(\wtil{x_j} + f\wtil{I_j}) = i_j,$ therefore $\wtil{x_j} - \wtil{i_j} = q_jf,$ for some $q_j \in Q$. Using this and the fact that $\pi_j$s are $R$-linear maps, we get from equation \eqref{eqn_1} 
	\begin{align*}
		h(\alpha)(i_1i_2) & = i_2 g_1(\wtil{i_1} + q_1f + f\wtil{I_1}) - i_1 g_2(\wtil{i_2} + q_2f + f\wtil{I_2}) \\ & = \wtil{i_2}\wtil{i_1} - \wtil{i_1}\wtil{i_2} + f\wtil{I_1}  + f\wtil{I_2} = 0.
	\end{align*} Thus $\alpha \in Ker(h)$. Hence by exact sequence \eqref{eqn_2}, there is an element $\beta \in  \Hom_R(I_1 + I_2,E )$ such that $h'(\beta) = \alpha$. Unraveling the maps, we get 
	\begin{align} \label{eqn_beta} 
		\beta: I_1 + I_2 \rar E , \text{ defined as }  \beta(i_1 + i_2) = g_1 \circ \pi_1(i_1) + g_2 \circ \pi_2(i_2).
	\end{align} In fact, the vanishing of $h(\alpha)$ shows that the maps $g_j \circ \pi_j: I_j \rar E )$ agree on $I_1 \cap I_2,$ hence glue together to give the map $\beta$.

	We next claim that the map $\beta$ splits the sequence
	\begin{align*}
		\eta: \ \ 0 \rar R/(I_1+I_2) \rar E  \xrightarrow{p} I_1 + I_2 \rar 0.
	\end{align*}

By construction, \(p\circ g_j\circ\pi_j\) is the inclusion \(I_j\hookrightarrow I_1+I_2\) for $j = 1,2$. Hence \(p\circ\beta=\operatorname{id}_{I_1+I_2}\), and \(\beta\) splits \(\eta\), as claimed.
%
%
%
%
\end{proof}

The following example, originally due to Hochster \cite{Ho}, shows that the assumption on transverse intersection of ideals in Theorem~\ref{WL_tensor product_cyclic} cannot be omitted.

\begin{exmp} \label{Example_Dao}
	Consider the ring of formal power series $Q= \Z _{(2)} [[u,v,w,x,y,z]]$. Let $R=Q/(f)$, where $f=2$. Consider the following ideals of $Q$: $\wtil{I}_1= (f,u^2, v^2, w^2,x^2,y^2,z^2)$, $\wtil{I}_2= (f, ux+vy+wz),$ and $\wtil{I}_3= (f,ux+vy+wz, u^2, v^2, w^2,x^2,y^2,z^2).$ For each $j=1,2,3$, set $I_j = \wtil{I}_j/(f)$. Since $I_1$ and $I_2$ are complete intersection ideals of $R$,  $R/I_1$ and $R/I_2$ lift, in particular, weakly lift, to $Q$. On the other hand, Dao \cite[Example~3.5]{Dao} showed that $R/{I}_3$ is not weakly liftable to $Q$. Since $R/{I}_3  \cong R/I_1 \otimes_R R/I_2$, it follows that the $R$-module $R/I_1 \otimes_R R/I_2$ does not weakly lift. One can easily check that $\Tor_1 ^R (R/I_1 , R/I_2)  \neq 0$.
\end{exmp}
	
	It is natural to ask if the analogous version of \cite[Proposition~2.4]{J1} holds for the weak liftable case. We prove it under some additional assumptions. 
	\begin{proposition}\label{WL_tensor_product}
		Let $M_1$ and $M_2$ be $R$-modules. Assume that $M_1$ is weakly liftable to $Q$, and that $\Tor_1 ^R (M_1,M_2)=0$. Then $M_1 \otimes_R M_2$ is weakly liftable provided at least one of the following two conditions hold.
		\begin{enumerate}[label=\normalfont(\alph*)]
			\item $M_2$ is liftable to $Q$.
			\item $M_2$ is weakly liftable to $Q$ and $\Tor_2 ^R (M_1,M_2)=0$.
		\end{enumerate}  
	\end{proposition}
	
	\begin{proof}
		Assume first the hypotheses in (b). Since both $M_1$ and $M_2$ are weakly liftable, by characterization \ref{WL_split map} in Section \ref{sec_WL}, we get $M_i \oplus \Omega_1 ^R M_i \cong {\Omega_1 ^Q M_i} \otimes_Q R$ for $i = 1,2$. This shows that $M_1 \otimes M_2$ is a direct summand of $ {\Omega_1 ^Q M_1 \otimes_{Q} \Omega_1 ^Q M_2 } \otimes_Q R$. Thus, it is sufficient to show that the latter is a liftable $R$-module, equivalently, that $f$ is a regular element on ${\Omega_1 ^Q M_1 \otimes_{Q} \Omega_1 ^Q M_2}$. 
		
		Consider the sequence $0 \rar \Omega_1 ^Q M_1 \rar \wtil{F}_1 \rar M_1 \rar 0$, where $\wtil{F}_1$ is a free $Q$-module. Tensoring it with $\Omega_1 ^Q M_2$ induces a long exact sequence $\Tor_1 ^Q (\Omega_1 ^Q M_2,M_1) \rar \Omega_1 ^Q M_2 \otimes_{Q} \Omega_1 ^Q M_1 \rar \Omega_1 ^Q M_2 \otimes_{Q} \wtil{F}_1$. Now the claim follows as by hypothesis 
		\[
		\Tor_1 ^Q (\Omega_1 ^Q M_2,M_1)  \cong \Tor_1 ^R ({\Omega_1 ^Q M_2 \otimes_QR},M_1) = \Tor_1 ^R ({M_2 \oplus \Omega_1 ^R M_2},M_1) = 0.
		\] 
		
		The proof of part (a) follows similarly. 
	\end{proof}

\section{Lifting of Auslander transposes and torsionlessness}
Throughout this section, unless otherwise specified, we assume that $Q$ is a local ring, $f$ is a $Q$-regular element, and $R = Q/(f)$. Let $M$ be an $R$-module  that lifts to a $Q$-module $N$. We follow the conventions from section \ref{sec_auslander}. In particular, since minimal free resolutions over local rings are determined uniquely up to isomorphisms, by regularity of $f,$ any minimal free $Q$-resolution $\F_{\bullet} \rar N,$ gives a minimal free $R$-resolution $\F_{\bullet}  \otimes_Q R \rar M$. Thus, we assume that \textit{both} Auslander transpose modules - $\Tr_R(M)$ and $\Tr_Q(N)$ are determined uniquely up to isomorphisms. 

It follows from the definition of liftability that $f$ is regular on $N,$ but in general, $f$ fails to be regular on $Tr_Q(N)$. This motivates the following result. 

\begin{lemma}\label{lemma_Auslander_transpose_lift} \label{torsionfree}
	The following are equivalent:
	\begin{enumerate}[label=\normalfont(\alph*)]
		\item \label{item_1} $f$ is regular on $\Tr_Q(N)$.
		\item \label{item_2} $\Tr_Q(N)$ is a lift of $\Tr_R(M)$.
		\item \label{item_3} $\Hom_Q(N,Q)$ is a lift of $\Hom_R(M,R)$ via the natural map 
		$$\theta: \Hom_Q(N,Q) \otimes_Q R \rar \Hom_R(M,R)$$ 
		\item \label{item_4} Either $\Ext^1_Q(N,Q) = 0$ or $f$ is regular on $\Ext^1_Q(N,Q)$.
	\end{enumerate} 
\end{lemma}
\begin{proof} Choose a minimal free presentation $F_1\xrightarrow d F_0\to N\to0$. Assume that $f$ is regular on $\Tr_Q(N)$. There is a short exact sequence $$0 \rar \Tr_Q(N) \xrightarrow{f} \Tr_Q(N) \rar \Tr_Q(N)/f\Tr_Q(N) \rar 0.$$ By \cite[Lemma~4.8]{AB}, $\Tr_Q(N)/f\Tr_Q(N) \cong \Tr_R(M)$. This proves	\ref{item_1} $\implies$ \ref{item_2}. 
	
The implication	\ref{item_2} $\implies$ \ref{item_3} follows by applying $\otimes_Q R$ to the following exact sequence
\begin{align} \label{seq_1} 
	0\rar N^* \rar F_0^* \xrightarrow{d^*}F_1^* \rar \Tr_Q(N) \rar 0.
\end{align}  
Set $K=\im(d^*)$. Since $K\subseteq F_1^*$, the element $f$ is regular on $K$; by \ref{item_2}, it is also regular on $\Tr_Q(N)$. Breaking the exact sequence \ref{seq_1} into two short exact sequences with intermediate module $K$, and tensoring one of them with $R$ yields the following exact sequence
$$0 \rar N^*\otimes_QR \rar F_0^*\otimes_QR \xrightarrow{d^*\otimes R}F_1^*\otimes_QR.$$ 
The kernel of the last map in this sequence is $Hom_R(M,R)$. The induced identification $$Hom_Q(N,Q) \otimes_Q R \rar Hom_R(M,R)$$ 
is precisely the canonical map $\theta$, hence it is an isomorphism. Since $f$ is regular on $Hom_Q(N,Q)$, therefore \ref{item_3} follows.
	
Assuming \ref{item_3}, we consider the long exact sequence $$0 \rar \Hom_Q(N,Q) \xrightarrow{f} \Hom_Q(N,Q) \rar \Hom_Q(N,R) \xrightarrow{\partial} \Ext^1_Q(N,Q) \xrightarrow{f}  \Ext^1_Q(N,Q).$$ Since $\Hom_Q(N,R) \cong \Hom_R(M,R),$ the hypothesis implies that the boundary map $\partial$ is zero, thus proving \ref{item_4}.
	
To complete the chain of implications, we assume \ref{item_4} to get a short exact sequence $$0 \rar \Hom_Q(N,Q) \xrightarrow{f}  \Hom_Q(N,Q) \rar \Hom_Q(N,R) \rar 0.$$ Tensoring this with $R,$ gives an isomorphism $\Hom_Q(N,Q) \otimes_Q R \cong \Hom_Q(N,R)$. The 4-term sequence \cite[Theorem~2.8]{AB} shows that  $\Tor_1^Q(Tr_Q(N), R) = 0,$ which proves \ref{item_1}.

\end{proof}

An immediate consequence of Lemma~\ref{torsionfree} follows.

\begin{corollary}\label{cor_syzygy_Sk}
	Let $M$ be an $R$-module  that lifts to a $Q$-module $N$. For any $n \geq 0,$ let $\partial_n$ be the boundary map $$\Ext_R^n(Tr_R(M),R) \xrightarrow{\partial_n} \Ext_Q^{n+1}(Tr_Q(N),Q).$$  If any of the equivalent conditions in Lemma~\ref{lemma_Auslander_transpose_lift} is satisfied, then the following are equivalent.
		\begin{enumerate}[label=\normalfont(\alph*)]
			\item \label{item_5} $M$ is $n$-torsionless.
			\item \label{item_6} $N$ is $n$-torsionless, and the map $\partial_{n}$ is zero. 
		\end{enumerate}
		Furthermore, if either $\Gdim_R (M)$ or $\Gdim_R (M^*)$ is finite, then these conditions are equivalent to the following
		\begin{enumerate}[label=\normalfont(\alph*),resume]
			\item \label{item_5'} 	 $M$ is an $n^{th}$ syzygy.
			\item \label{item_5''}  $M$ satisfies $\wtil{S}_n$. 
			\item \label{item_6'} $N$ is an $n^{th}$ syzygy, and  the map $\partial_{n}$ is zero. 
			\item \label{item_6''} $N$ satisfies $\wtil{S}_n,$ and the map $\partial_{n}$ is zero. 
		\end{enumerate}
\end{corollary}

\begin{proof} Assume that any of the equivalent conditions in Lemma~\ref{lemma_Auslander_transpose_lift} holds. 
	Applying $\Hom_Q(\Tr_Q(N), -)$ to $0 \rar Q \xrightarrow{f} Q \rar R \rar 0$ and using $\Ext_Q^i(\Tr_Q(N), R) \cong \Ext_R^i(\Tr_R(M), R)$, 
	we obtain the following exact sequence. \begin{align*}
		\rar \Ext_Q^n(Tr_Q(N),Q) \xrightarrow{f} \Ext_Q^n(Tr_Q(N),Q) \rar \Ext_R^n(Tr_R(M),R) \xrightarrow{\partial_n} \Ext_Q^{n+1}(Tr_Q(N),Q)
	\end{align*} The equivalence of $\ref{item_5}$ and $\ref{item_6}$ follows by Nakayama’s lemma.

 If $\Gdim_R (M) < \infty,$ it follows by \cite[Corollary~4.30]{AB} that $\Gdim_Q (N) < \infty$. If $\Gdim_R (M^*)  < \infty,$ it follows by Lemma~\ref{lemma_Auslander_transpose_lift} and \cite[Corollary~4.30]{AB} that $\Gdim_Q (N^*)  < \infty$. Under these finiteness assumptions on the Gorenstein dimensions, the equivalence of \ref{item_5}, \ref{item_5'} and \ref{item_5''} and the equivalence of \ref{item_6}, \ref{item_6'}, and \ref{item_6''} is established in \cite{AB} and \cite{VV}. 
\end{proof}

In Corollary~\ref{cor_syzygy_Sk}, the mere assumption that $N$ is $n$-torsionless is not sufficient to obtain the same conclusion for $M$, as Example~\ref{ex_2} demonstrates. Example~\ref{ex_1} shows that the hypothesis that $M$ or $M^*$ has finite Gorenstein dimension in the last assertion
 is also essential.  To preserve the flow of the exposition, we defer both examples to Section~\ref{Sec_exmp}.


\subsection{A freeness criterion} For the rest of this section, $R$ denotes an arbitrary local ring. As another application of Lemma~\ref{torsionfree}, we prove a special case of the following conjecture. 

\begin{conj}[Auslander--Reiten conjecture]\label{AR_conj}
	Let $R$ be a local ring, and $M$ an $R$-module. If \[\Ext_R ^i(M,R)= \Ext_R ^i(M,M) = 0 \quad \text{ for all  }  i \geq 1,\] then $M$ is free.
\end{conj} 

Conjecture \ref{AR_conj} remains open in general, although it has been shown to hold true for several cases. We refer to  \cite{Ar,ADS,Kim} for some results and further references. More recently, the following result has been proved. 

\begin{theorem}[{\cite[Theorem~3.10]{GS}}]\label{GS_AR_Conj}
	Let $M$ be an $R \m{-}$module such that $CI\m{-}dim_{R} \Hom(M,M) < \infty$. Suppose that $\Ext^{i} _{R} (M,R) =0$ for $1 \leq i \leq \depth(R),$ $\Ext^{i} _{R} (M,M) =0$ for all $i \geq 1$, and at least one of $G\mbox{-}dim_R (M)$ or $G\mbox{-}dim_R (M^*)$ is finite. Then $M$ is free.
\end{theorem}

We prove the following  result, whose statement we recall from the introduction. This result recovers
Theorem~\ref{GS_AR_Conj}.

\begin{recalledB}
	Let $M$ be an $R$-module such that $CI\m{-}dim_{R} \Hom(M,M) < \infty$. Assume that $\Ext^{1} _{R} (M,R) =0$ and $\Ext^{i} _{R} (M,M) =0$ for all $i \geq 1$. If $M$ is $n$-torsionless for all positive integers $n$, then $M$ is free.
\end{recalledB}

\begin{proof} We follow the quasi-deformation argument in the proof of \cite[Theorem~3.10]{GS}, using Lemma \ref{lemma_Auslander_transpose_lift} and Corollary \ref{cor_syzygy_Sk} to handle the torsionlessness hypothesis.

Since $CI\m{-}dim_{R} \Hom(M,M) < \infty$, there exists a quasi-deformation $R \rar R' \twoheadleftarrow T$ such that $pd_{T} (\Hom_R (M,M) \otimes_R R') < \infty$. By flat base change and completion, we may assume that $T$ and $R$ are complete, that $R=T/(\mathbf f)$ for a $T$-regular sequence $\mathbf f=f_1,\ldots,f_t$, and that
$$\pd_T\Hom_R(M,M)<\infty.$$ 
	The remaining hypotheses are preserved under these extensions, and freeness descends.
	
	 Since $\Ext_R^2(M,M)=0$, it follows by \cite[Proposition~1.7]{ADS} that there is a $T$-module $N$ that lifts $M$. Thus, $\mathbf{f}$ is regular on $N$, and $N/\mathbf{f}N \cong M$. Set $$T_j=T/(f_1,\ldots,f_j),\qquad N_j=N/(f_1,\ldots,f_j)N \quad(0\leq j\leq t).$$ 
	 In particular, $(T_0,N_0)=(T,N)$ and $(T_t,N_t)\cong(R,M)$. Repeated application of \cite[Lemma~3.1.16]{BH} and Nakayama’s lemma, working backwards from \(j=t\), gives
	$$\Ext_{T_j}^{i}(N_j,N_j)=0 \qquad(i\geq1,\;0\leq j\leq t).$$ 
	 Writing $E_j=\operatorname{Hom}_{T_j}(N_j,N_j)$, the associated exact sequences yield
	 $$
	 	0\rar E_{j-1}	\xrightarrow{\,f_j\,}E_{j-1} 	\rar E_j\rar 0	\qquad(1\leq j\leq t).$$
	 	
	 	Consequently, \(\mathbf f\) is $E_0$-regular and $E_0/(\mathbf f)E_0 \cong E_t =\Hom_R(M,M)$. It follows that
	 	$$\pd_{T} (\Hom_T (N,N)) < \infty.$$

	Similarly, the hypothesis $\Ext_R^1(M,R)=0$ gives
	$$
	\Ext_{T_j}^1(N_j,T_j)=0 \qquad(0\leq j\leq t).$$
	Applying Lemma~\ref{torsionfree} successively to \(T_{j-1}\to T_j\), we obtain that
	$\Tr_{T_{j-1}}(N_{j-1})$ lifts
	$\Tr_{T_j}(N_j)$. In particular, $Tr_T(N)$ is a lift of $Tr_R(M)$. Since $N_t\cong M$ is $n$-torsionless for every positive integer $n$, successive applications of Corollary~\ref{cor_syzygy_Sk}$ \ref{item_5} \implies \ref{item_6}$ show that the same holds for $N$, i.e., $\Ext^{n} _{T} (\Tr_T(N),T) =0$ for all $n \geq 1$. The claim now follows from \cite[Theorem~6.2.(1)]{DG} applied with both modules equal to $N$. 
\end{proof}

\begin{proof}[Proof of Theorem~\ref{GS_AR_Conj} using Theorem~\ref{Auslander-Reiten_torsionfree}]
	Under the Ext-vanishing hypotheses of Theorem~\ref{GS_AR_Conj}, if $G\mbox{-}dim_R (M) < \infty,$ then by \cite[Proposition~4.12]{AB}, the module $M$ is totally reflexive. If instead  $G\mbox{-}dim_R (M^*) < \infty$, then \cite[Theorem~4.1]{VV} yields the same conclusion. Thus, in either case, $M$ is $n$-torsionless for all $n \geq 1$. We apply Theorem~\ref{Auslander-Reiten_torsionfree} to complete the proof. 
\end{proof}

\section{Examples} \label{Sec_exmp} 
In this section we discuss some examples. Some arguments here rely on computations performed in Macaulay2. The supplementary file \textit{Liftability.m2} contains the computations for Theorem~\ref{thm_C} and Example~\ref{ex_1}.  The script has been commented to make it easy to identify the computations corresponding to each result here.

\subsection{Serre liftability without weak liftability} 

We begin by describing an example due to Gasharov and Peeva \cite{GP} and later studied by Jorgensen and \c{S}ega \cite{JS}. 

\begin{exmp} \label{exmp_main} Let $k$ be a field and let $0 \neq \alpha\in k.$ Let $Q = k[[X_1,\cdots, X_5]]$ and let $I_{\alpha}$ be the ideal generated by the following elements \begin{align*}
		\alpha X_1X_3 + X_2 X_3, \ X_1X_4 + X_2X_4, \ X_3^2 - X_2X_5 + \alpha X_1 X_5, \\
		\ \\ 
		X_4^2 - X_2 X_5 + X_1 X_5, \ X_1^2, \ X_2^2, \ X_3X_4,\ X_3X_5, \ X_4X_5, \ X_5^2.
	\end{align*} Set $f = X_1^2$ and $R = Q/(f)$. 
\end{exmp}

We now use the preceding example to prove Theorem~\ref{thm_C}, whose statement we reproduce here for convenience.

\begin{recalledC}
		There exists a perfect cyclic module of codimension $4$ over a hypersurface singularity that admits a Serre lift but fails to be weakly liftable to the ambient regular local ring. The module can be chosen to be an Artinian Gorenstein quotient of a $5$-dimensional complete regular local ring.
\end{recalledC}
\begin{proof}
	We consider Example \ref{exmp_main} with $k = \Q$. It can be easily checked that $I_{\alpha}$ is $\fm$-primary where $\fm = (X_1,\cdots, X_5).$ In particular, $\height(I_{\alpha}) = 5$. Let $J_{\alpha}$ be the ideal generated by the same generators as $I_{\alpha}$ except $f$. Then $\sqrt{J_{\alpha}} =  (X_2, X_3, X_4, X_5)$. Thus, $\height(I_{\alpha}) - \height(J_{\alpha}) = 1.$ This shows, by the characterization obtained in Lemma~\ref{lemma_serre_gen_equiv}, that the $Q$-module $Q/J_{\alpha}$ is a Serre lift of the cyclic $R$-module $Q/I_{\alpha}$. 
	
	It can be verified that $m^4 \subset I_{\alpha} \subset m^2$. We note that $Q/I_{\alpha}$ is Gorenstein, and the socle $(\fm^3 + I_{\alpha})/I_{\alpha}$ is generated by the image of $g = X_1X_2X_5$; see \cite{GP}. It follows from this that $g \notin I_{\alpha}$. By considering $\alpha$ as an independent variable, a Macaulay2 computation shows the following inclusion of ideals \begin{align} \label{eqn_containment} 
		(\alpha^2 - 1) \subseteq (I_{\alpha}^2: fg)
	\end{align}Thus, we conclude that when $\alpha \neq \pm 1$, there is a containment $g \in  (I_{\alpha}^2: f) \setminus I_{\alpha}$. By Dao's criterion for weak lifting \ref{sec_WL}\ref{Dao_criterion}, it follows that $Q/I_{\alpha}$ is not weakly liftable.
		
Specializing to $\alpha = 2$, a computation in Macaulay2 verifies that $Q/I_2$ is a perfect $R$-module. Direct verifications of the two computational steps used here are given in the following remark.
\end{proof}

\begin{remark}[Direct verification of the computations] \label{thm_C_AI}
	The Macaulay2 computations in the preceding proof admit the following direct verifications that were suggested by ChatGPT (OpenAI) and independently verified by the authors. Take $u = X_1 + X_2$ and $h = X_4^2 - X_2 X_5 + X_1 X_5$. Then, 
	\[
	\begin{aligned}
		X_1u^2h
		&=X_1(uX_4)^2+X_1X_5u^2(X_1-X_2)\\
		&=X_1(uX_4)^2
		+X_5(X_1^2)^2
		- X_5(X_1^2)(X_2^2)\\
		&\qquad
		+X_1^3X_2X_5-X_1X_2^3X_5.
	\end{aligned}
	\]
	
	Since $X_1u^2\in(X_1^2,X_2^2)\subseteq I_\alpha$ and $h\in I_\alpha$, the left-hand side belongs to \(I_\alpha^2\); the first three terms on the right do as well. Therefore,
	\begin{align} \label{eqn_3} 
			X_1u^2h	\equiv
		X_1^3X_2X_5-X_1X_2^3X_5 \equiv 0
		\pmod{I_{\alpha}^2}.
	\end{align}

	Similarly, take $v = \alpha X_1 + X_2$ and $h' = X_3^2 - X_2X_5 + \alpha X_1 X_5$. Then, 
	\begin{align} \label{eqn_4} 
	\nonumber	X_1v^2 h' & = X_1 (vX_3)^2 + X_1 X_5 v^2 (\alpha X_1 - X_2) \\ 
		& \equiv \alpha^2X_1^3 X_2 X_5 - X_1X_2^3X_5 \pmod{I_{\alpha}^2}. 
	\end{align}
	Since $X_1 v^2, h' \in I_{\alpha}$, the left hand side belongs to $I_{\alpha}^2$, we get on subtracting equations \eqref{eqn_3} and \eqref{eqn_4} that
$$
	(\alpha^2-1)fg\in I_\alpha^2,
$$ which is the containment \eqref{eqn_containment} of Theorem \ref{thm_C}.  
	
	For perfection, set \(A=Q/I_\alpha\), and write \(x_i\) for the image of \(X_i\) in \(A\). By \cite[Section~2]{JS}, a \(k\)-basis of \(A\) is
	\[
	\mathcal B\cup x_1\mathcal B,
	\qquad
	\mathcal B=\{1,x_2,x_3,x_4,x_5,x_2x_5\}.
	\]
	Thus \((0:_A x_1)=x_1A\). 
	Set \(S=R/(X_1)\cong k[[X_2,\ldots,X_5]]\). The periodic $R$-free resolution
	\[
	\cdots\xrightarrow{\,X_1\,}R
	\xrightarrow{\,X_1\,}R
	\xrightarrow{\,X_1\,}R
	\longrightarrow S\longrightarrow0
	\]
	shows that $\Tor_i^R(A,S)=0$ for all $i>0$. Consequently, reduction modulo $X_1$ takes a minimal $R$-free resolution of \(A\) to a minimal \(S\)-free resolution of $A/x_1A$. Since $S$ is regular of dimension \(4\) and \(A/x_1A\) has finite length,
	\[
	\operatorname{pd}_R A
	=\operatorname{pd}_S(A/x_1A)=4.
	\]
	Therefore \(A\) is a perfect $R$-module.
\end{remark}

\begin{remark} \label{rem_6}
	Suppose that $Q$ is a regular local ring and that $Q/I$ is a Gorenstein quotient of codimension at most $4$ with finite projective dimension over $R$, as in Example~\ref{exmp_main}. Then $Q/I$ lifts; see \cite[p.~478]{BE} and \cite[p.~501]{J1}. Thus, the example presented above has the smallest possible codimension over $Q$ among Gorenstein quotients having finite projective dimension over $R$ that fail to lift to $Q$.
\end{remark}

Another example in higher codimension is given by Hochster's well-known counterexample \cite{Ho} to Grothendieck's lifting question. 
\begin{exmp} \label{exmp_main2} As in Example \ref{Example_Dao}, let $Q= \Z _{(2)} [[u,v,w,x,y,z]]$ and consider the ideal $\wtil{I}_3= (f,ux+vy+wz, u^2, v^2, w^2,x^2,y^2,z^2)$ where $f = 2$. We set $R = Q/(f)$ and define $$\wtil{J}_3 = (ux+vy+wz, u^2, v^2, w^2,x^2,y^2,z^2).$$ It can be verified that $\height(\wtil{I}_3)  = \height(\wtil{J}_3) +  1 = 7$. By Lemma~\ref{lemma_serre_gen_equiv}, $Q/\wtil{J}_3$ is a Serre lift of the $R$-module $Q/\wtil{I}_3$. However, $Q/\wtil{I}_3$ is not weakly liftable as shown in \cite[Example~3.5]{Dao}. 
\end{exmp}

\begin{remark}
	Serre liftability of a related ramified example of Hochster is noted in \cite[Example~1.7]{NL}; see also \cite[Example~2.3.6]{NKC}. For the case \(f=2\) considered here, Lemma \ref{lemma_serre_gen_equiv} gives a direct verification. Combining this with Dao’s non-weak-liftability result gives another example separating Serre and weak liftability.
\end{remark}

\subsection{Examples for Corollary 11} 

We first present an example that shows that assuming only that $N$ is $n$-torsionless does not suffice to obtain the conclusion of Corollary~\ref{cor_syzygy_Sk} for $M$.
\begin{exmp} \label{ex_2} 
Let $Q=k[[x,y,z]]$ be a power series ring in three  indeterminates over the field $k$, and let $N$ be its maximal ideal $\fn$. 	It is easy to check that the element $f = x$ is regular on $N \oplus Q$. Set $R=Q/(f)$ and $M=N/fN$.  Since $N$ embeds in $Q$, it is 1-torsionless, i.e., $\Ext^1_Q(\Tr_Q(N), Q) = 0.$  Moreover, $\Ext_Q^1(N,Q) \cong \Ext^2_Q(k, Q) = 0$, so the equivalent conditions of Lemma \ref{lemma_Auslander_transpose_lift} are satisfied.
	
Since $\pd_Q(N)=2$, and a minimal $Q$-free resolution of $N$ remains exact and minimal after tensoring with $R$, we have $\pd_R(M)=2$, hence, by the Auslander–Buchsbaum formula, $\depth(M)=0$, which shows that $M$ is not $1$-torsionless, i.e., $\Ext^1_R(Tr_R(M), R) \neq 0$. In particular, the map $\partial_1$ of Corollary~\ref{cor_syzygy_Sk} is nonzero.

\end{exmp}

The following example shows that the hypothesis $M$ or $M^*$ has finite Gorenstein dimension in the last assertion of the Corollary~\ref{cor_syzygy_Sk} is essential. 

\begin{exmp} \label{ex_1} 	Let $Q= k[[X_1,X_2,X_3]]/(X_1X_3 ,X_2X_3)$ and let $x_i$ be the image of $X_i$ in $Q$ for $i=1,2,3$. Consider the $Q$-module $N= Q/(x_2 ^2)$ and set $\fp = (x_2, x_3)$. Then $\fp$ is a minimal prime of $N$, hence $\depth (N_{\fp} )= 0$. On the other hand, $\fp$ is a nonsingular point of the variety defined by $Q$ so $ \depth(Q_{\fp}) = \dim(Q_{\fp}) = 1$. So $N$ does not satisfy $\wtil{S}_1$. 
	
We set $f = x_1 + x_3$ and $R = Q/(f).$ It is clear that $f$ is not in the associated primes of $Q$ or $N$. It is therefore a nonzerodivisor on $Q \oplus N$. Let $M = N/fN$ be the quotient $R$-module.  Since $\Tr_Q(N)\cong N$, the equivalent conditions of Lemma \ref{lemma_Auslander_transpose_lift} are also satisfied. 

The computations in the Macaulay2 code included in the supplementary material verifies that 
\[
\operatorname{depth}R=\operatorname{depth}M
=\operatorname{depth}M^*=0,
\qquad
\operatorname{Ext}_R^1(M,R)\neq0,
\quad
\operatorname{Ext}_R^1(M^*,R)\neq0.
\]
The Auslander–Bridger formula therefore implies that both \(M\) and \(M^*\) have infinite Gorenstein dimension. Moreover, \(\operatorname{Supp}M=\{\mathfrak m\}\), so \(M\) satisfies \(\widetilde S_1\). Thus the finiteness hypothesis in the last assertion of Corollary ~\ref{cor_syzygy_Sk} cannot be omitted.

\end{exmp}

\subsection*{Acknowledgements} 

The first author was partially supported by a CSIR senior research fellowship through the grant no. 09/1001(0097)/2021-EMR-I. Macaulay2~\cite{M2} was used to carry out computations required for examples in this paper.

After completing the first version, we used ChatGPT (OpenAI) to obtain feedback on the manuscript that included suggestions on presentation and on possible issues in statements and arguments. The direct verifications in Remark \ref{thm_C_AI} were also suggested during this process. We independently checked all suggestions incorporated into the paper and take full responsibility for the final manuscript.

\end{document}